\documentclass[12pt,a4paper]{article}
\usepackage{amssymb,amsmath}
\usepackage{mathrsfs}
\usepackage{tikz}
\usepackage{graphicx}
\usepackage{amsthm}

\theoremstyle{definition} 
\newtheorem{df}{Definition}[section] 
   
\theoremstyle{plain}            
\newtheorem{pro}[df]{Proposition}
\newtheorem{lem}[df]{Lemma}
\newtheorem{theo}[df]{Theorem}
\newtheorem{cor}[df]{Corollary}

\newcommand{\f}{\ensuremath{\varphi}}

\newcommand{\al}{\ensuremath{\alpha}}

\newcommand{\la}{\ensuremath{\lambda}}

\newcommand{\Dcal}{\ensuremath{\mathcal{D}}}
\newcommand{\Ecal}{\ensuremath{\mathcal{E}}}

\newcommand{\Mcal}{\ensuremath{\mathcal{M}}}
\newcommand{\Ncal}{\ensuremath{\mathcal{N}}}

\newcommand{\Zcal}{\ensuremath{\mathcal{Z}}}

\newcommand{\cc}{\ensuremath{\mathbb{C}}}

\newcommand{\rr}{\ensuremath{\mathbb{R}}}

\newcommand{\unit}{\ensuremath{\mathbf{1}}}
\newcommand{\norm}[1]{\ensuremath{\left\|#1\right\|}}

\newcommand{\set}[2]{\left\{#1\,\middle|\,#2 \right\}}

\newcommand{\Ca}{$C${\rm*}-algebra}      
\newcommand{\Csa}{$C${\rm*}-subalgebra}

\newcommand{\vNa}{von Neumann algebra}

\newcommand{\AW}{$AW${\rm*}-algebra}
\newcommand{\AWf}{$AW${\rm*}-factor}
\newcommand{\AWsa}{$AW${\rm*}-subalgebra}

\newcommand{\Sa}{{\rm*}-algebra}

\newcommand{\Sim}{{\rm*}-isomorphism}

\newcommand{\Sic}{{\rm*}-isomorphic}
\newcommand{\Jsi}{Jordan {\rm*}-isomor\-phism}

\newcommand{\ifff}{if and only if}

\begin{document}

\begin{center}
{\Large \bf Isomorphisms of spectral lattices}\\

{\large Martin Bohata\footnote{martin.bohata@fel.cvut.cz}\\}
         \it Department of Mathematics, Faculty of Electrical Engineering,\\
        Czech Technical University in Prague, Technick\'a 2,\\ 
        166 27 Prague 6, Czech Republic        
\end{center}

{\small \textbf{Abstract:} The paper deals with spectral order isomorphisms between certain spectral sublattices of direct sums of \AWf{}s. We prove that these maps consist of spectral order isomorphisms between spectral sublattices of individual direct summands. Consequently, we obtain a complete description of spectral order isomorphisms in the case of atomic \AW{}s. This includes the setting of matrix algebras. Moreover, we also exhibit the general form of spectral order orthoisomorphisms between various spectral sublattices of direct sums of \AWf{}s.} 



\section{Introduction}

Let $(E^x_\la)_{\la\in\rr}$ and $(E^y_\la)_{\la\in\rr}$ be spectral families of self-adjoint elements $x$ and $y$, respectively, in an \AW{} \Mcal{}. We write $x\preceq y$ if $E^{y}_\la\leq E^x_\la$ for all $\la\in\rr$. The binary relation $\preceq$ on the self-adjoint part of \Mcal{} is a partial order called \textit{spectral order}. It was first introduced by Olson \cite{Ol71} in the setting of \vNa{}s. The self-adjoint part $\Mcal_{sa}$ of \Mcal{} endowed with the spectral order forms conditionally complete lattice, statement which can be proved in the same way as for \vNa{}s \cite{Ol71}. This result is in a strong contrast to the behavior of the standard order on self-adjoint elements \cite{Ka51,Sh51}. In the sequel, we shall call $(\Mcal_{sa},\preceq)$ the {\it spectral lattice} of \Mcal{} following the terminology introduced in \cite{dG05}. By a {\it spectral sublattice} of \Mcal{} we shall mean a sublattice of the spectral lattice of \Mcal{}. It follows from \cite[Proposition~3.4]{Bo21} that examples of proper spectral sublattices are $(\Mcal_+,\preceq)$, $(\Ecal(\Mcal),\preceq)$, and $(P(\Mcal),\preceq)$, where $\Mcal_+$ is the positive part of \Mcal{}, $\Ecal(\Mcal)$ is the set of all effects (i.e. the set of all positive elements in the unit ball of \Mcal{}), and $P(\Mcal{})$ is the set of all projections in \Mcal{}. It is easy to see that the spectral order coincides with the standard order on projections. Consequently, the spectral lattice can be regarded as a natural extension of the projection lattice of \Mcal{}.

It is worth to note that the spectral order has the following physical interpretation. Let $w((-\infty,\la],x,\f)$ be the probability that a measurement of an observable $x$ gives a value in an interval $(-\infty,\la]$ in a state $\f$ of a physical system. Then $x\preceq y$ says that $w((-\infty,\la],y,\f)\leq w((-\infty,\la],x,\f)$ for every $\la\in\rr$ and every state $\f$. Thus $x\preceq y$ means that, in every state of the physical system, the corresponding distribution functions are pointwise ordered. The interested reader can find physical applications of the spectral order, for example, in the papers \cite{DD14,Ha11,Wo14}.

Let $M$ and $N$ be subsets of self-adjoint parts of two \AW{}s. A bijection $\f:M\to N$ is called a {\it spectral order isomorphism} if, for all $x,y\in M$, $x\preceq y$ \ifff{} $\f(x)\preceq \f(y)$. This paper is devoted to the study of spectral order isomorphisms between certain spectral sublattices of \AW{}s. We continue the line of research initiated by Moln\'{a}r and \v{S}emrl in \cite{MS07}. Among other things they described the general form of spectral order automorphisms of the self-adjoint part of the \vNa{} $B(H)$ of all bounded operators on a finite-dimensional complex Hilbert space $H$ provided that $\dim H\geq 3$. The two-dimensional setting was later examined by Moln\'{a}r and Nagy in \cite{MN16}. The open problem of infinite-dimensional case has recently been solved by the author in \cite{Bo21}. More concretely, it has been shown that if \Mcal{} and \Ncal{} are \AWf{}s of Type~I (i.e. \Mcal{} and \Ncal{} are \Sic{} to $B(H)$ and $B(K)$, respectively, for some complex Hilbert spaces $H$ and $K$), then every spectral order isomorphism $\f:\Mcal_{sa}\to\Ncal_{sa}$ has the form $\f(x)=\Theta_\tau(f(x))$, where $f:\rr\to \rr$ is strictly increasing bijection, $\tau$ is an isomorphism between projection lattices, and $\Theta_\tau$ is defined by $E_\la^{\Theta_\tau(x)}=\tau(E_\la^x)$ for all $\la \in\rr$. Similar statements are also known for spectral order isomorphisms between various proper spectral sublattices of \AWf{}s of Type~I \cite{Bo21, MN16, MR21, MS07}. However, there are no results on the structure of general spectral order isomorphisms between spectral lattices (or proper spectral sublattices) going beyond \AWf{}s of Type~I. The aim of this paper is to contribute to fill that gap by investigating spectral order isomorphisms in the context of direct sums of \AWf{}s. 

A simple observation shows that the mapping $(x,y)\mapsto (x,y^3)$ is a spectral automorphism of the spectral lattice of the \AW{} $B(H)\oplus B(H)$ which has not the above canonical form $\Theta_\tau(f(x))$. On the other hand, we see that components are transformed in the canonical way. It turns out that the componentwise action of spectral isomorphisms $\Ecal(\Mcal)\to\Ecal(\Ncal)$, $\Mcal_{+}\to\Ncal_{+}$, and $\Mcal_{sa}\to\Ncal_{sa}$, where $\Mcal$ and $\Ncal$ are direct sums of \AWf{}s, is not exceptional. Indeed, we establish that such isomorphisms are determined by spectral order isomorphisms between the corresponding spectral sublattices of direct summands. This allows us to describe their general form in the case of atomic \AW{}s using results from \cite{Bo21,MN16,MS07}. In particular, we obtain a complete description of spectral order isomorphisms between spectral lattices of matrix algebras.

A spectral order isomorphism $\f:M\to N$ is called a {\it spectral order orthoisomorphism} if, for all $x,y\in M$, $xy=0$ \ifff{} $\f(x)\f(y)=0$. The structure of spectral order orthoisomorphisms is well known in the case of \AWf{}s. Let \Mcal{} and \Ncal{} be \AWf{}s not of Type~I$_2$. In \cite{HT16,HT17}, Hamhalter and Turilova obtained a description of spectral order orthoautomorphism of $\Ecal(\Mcal)$ provided that an \AWf{} \Mcal{} is not of Type~III. However, the exclusion of \AWf{}s of Type~III is not needed. In fact, it has been proved in \cite{Bo21} that every spectral order orthoisomorphism $\f:\Ecal(\Mcal)\to\Ecal(\Ncal)$ has the form $\f(x)=\psi(f(x))$ for some \Jsi{} $\psi:\Mcal\to\Ncal$ and some strictly increasing bijection $f:[0,1]\to[0,1]$. It has also been established in \cite{Bo21} that analogous theorems hold for spectral order orthoisomorphisms $\f:\Mcal_+\to\Ncal_+$ and $\f:\Mcal_{sa}\to\Ncal_{sa}$.
In this paper, we use these results to get a general form of spectral order orthoisomorphisms of spectral sublattices of direct sums of \AWf{}s which are not of Type~I$_2$.

\section{Preliminaries}

We start this section by recalling some basic facts about \AW{}s. For a more detailed exposition of the theory of \AW{}s, we refer the reader to the monographs \cite{Be11,SW15,SS05}. An {\it \AW{}} is a 
\Ca{} \Mcal{} such that the following conditions hold:
\begin{enumerate}
	\item Every maximal commutative \Csa{} of \Mcal{} is a closed linear span of its projections.
	\item The set $P(\Mcal)$ of all projections equipped with the standard order is a complete lattice.
\end{enumerate}
Note that every \AW{} is unital. By the symbol $\unit_\Mcal$, we shall denote the unit of \Mcal{}. When no confusion can arise, we shall write $\unit$ in place of $\unit_\Mcal$. A \Csa{} \Ncal{} of an \AW{} \Mcal{} is called {\it \AWsa{}} of \Mcal{} if \Ncal{} is an \AW{} and the supremum of each family of orthogonal projections in \Ncal{} computed in the projection lattice of \Mcal{} is also an element of \Ncal{}. The {\it center} of an \AW{} \Mcal{} is the set 
\[
\Zcal(\Mcal)=\set{x\in\Mcal}{xy=yx \mbox{ for all }y\in \Mcal{}}.
\]
The center \Zcal(\Mcal) forms an \AWsa{} of \Mcal{}. An important class of \AW{}s consists of \vNa{}s. It was proved by Kaplansky \cite{Ka52} that an \AW{} of Type I is a \vNa{} \ifff{} its center is a \vNa{}. Consequently, \AWf{}s of Type I are nothing but von Neumann factors of Type I. On the other hand, it is well known that there are \AWf{}s which are not von Neumann factors \cite{Dy70, Ta78, Wr76}. 

Let $(\Mcal_j)_{j\in\Lambda}$ be a family of \AW{}s. Suppose that the set
$$
\bigoplus_{j\in\Lambda}\Mcal_j:=\set{(x_j)_{j\in\Lambda}}{\sup_{j\in\Lambda}\norm{x_j}<\infty}
$$
is equipped with the pointwise \Sa{} operations and the norm $\norm{(x_j)_{j\in\Lambda}}=\sup_{j\in\Lambda}\norm{x_j}$. Then it can be shown (see \cite{Be11}) that $\bigoplus_{j\in\Lambda}\Mcal_j$ forms an \AW{} called the {\it direct sum} of $(\Mcal_j)_{j\in\Lambda}$. A projection $p$ in an \AW{} is called {\it atomic}, if it has no nonzero proper subprojection. An \AW{} is said to be {\it atomic} if every nonzero projection majorizes an atomic projection. Note that each atomic \AW{} is (\Sic{} to) a direct sum of \AWf{}s of Type I. Thus atomic \AW{}s are precisely atomic \vNa{}s.

An extremally disconnected compact Hausdorff topological space is called {\it Stonean space}. It is well known that every abelian \AW{} is \Sic{} to $C(X)$, where $X$ is a Stonean space. Thus there is one-to-one correspondence between projection lattices of abelian \AW{}s and complete Boolean algebras. Accordingly, the projection lattice, $(P(\Mcal),\leq)$, of an abelian \AW{} \Mcal{} is meet-infinitely distributive (see \cite[Theorem 5.13]{Ro08}) which means that
\[
\bigwedge_{\la\in\Lambda}(p\vee q_\la)=p\vee \bigwedge_{\la\in\Lambda}q_\la
\]
for each $p\in P(\Mcal)$ and each family $(q_\la)_{\la\in\Lambda}$ in $P(\Mcal)$.

Recall that a family $(E_\la)_{\la\in\rr}$ of projections in an \AW{} $\Mcal$ is called a ({\it bounded}) {\it spectral family} if the following conditions hold:
\begin{enumerate}
	\item $E_\la\leq E_\mu$ whenever $\la\leq \mu$.
	\item $E_\la=\bigwedge_{\mu>\la}E_\mu$ for every $\la\in\rr$.
	\item There is a positive real number $\al$ such that $E_\la=0$ when $\la<-\al$ and $E_\la=\unit$ when $\la>\al$.
\end{enumerate}
It is part of the folklore of operator theory that there is a bijection between set of all spactral families in \Mcal{} and the self-adjoint part of \Mcal{}. The spectral family $(E_\la)_{\la\in\rr}$ corresponds to $x\in\Mcal_{sa}$ \ifff{} $xE_\la\leq \la E_\la$ and $\la(\unit-E_\la)\leq x(\unit-E_\la)$ for each $\la\in\rr$. In the sequel, we shall denote by $(E^x_\la)_{\al\in\rr}$ the spectral family corresponding to $x\in\Mcal_{sa}$. It turns out that $E^x_\la$ belongs to the abelian \AWsa{} of \Mcal{} generated by $\{\unit,x\}$.

The following two lemmas are well known. We present their proofs for the convenience of the reader.

\begin{lem}\label{spectral family of direct sum}
Let $(\Mcal_j)_{j\in\Lambda}$ be a family of \AW{}s and let $\Mcal=\bigoplus_{j\in\Lambda}\Mcal_j$. If $(x_j)_{j\in\Lambda}\in\Mcal_{sa}$, then $E^{(x_j)_{j\in\Lambda}}_\la=(E^{x_j}_\la)_{j\in\Lambda}$ for every $\la\in\rr$.
\end{lem}
\begin{proof}
Is easy to see that that $((E^{x_j}_\la)_{j\in\Lambda})_{\la\in\rr}$ is a spectral family. Since 
$$
(x_j)_{j\in\Lambda} (E^{x_j}_\la)_{j\in\Lambda}= (x_jE^{x_j}_\la)_{j\in\Lambda}\leq \la (E^{x_j}_\la)_{j\in\Lambda}
$$ 
and 
\begin{align*}
(x_j)_{j\in\Lambda}[\unit-(E^{x_j}_\la)_{j\in\Lambda}]
&= (x_j)_{j\in\Lambda}(\unit_{\Mcal_j}-E^{x_j}_\la)_{j\in\Lambda}
 =(x_j(\unit_{\Mcal_j}-E^{x_j}_\la))_{j\in\Lambda}\\
&\geq \la (\unit_{\Mcal_j}-E^{x_j}_\la)_{j\in\Lambda}
=\la [\unit-(E^{x_j}_\la)_{j\in\Lambda}]
\end{align*}
for every $\la\in\rr$, $E^{(x_j)_{j\in\Lambda}}_\la=(E^{x_j}_\la)_{j\in\Lambda}$ for every $\la\in\rr$.
\end{proof}

\begin{lem}\label{spectral order on direct sum}
Let $(\Mcal_j)_{j\in\Lambda}$ be a family of \AW{}s and let $\Mcal=\bigoplus_{j\in\Lambda}\Mcal_j$. If $(x_j)_{j\in\Lambda},(y_j)_{j\in\Lambda}\in\Mcal_{sa}$, then $(x_j)_{j\in\Lambda}\preceq (y_j)_{j\in\Lambda}$ \ifff{} $x_j\preceq y_j$ for every $j\in\Lambda$.
\end{lem}
\begin{proof}
Let $\la\in\rr$. It is easy to see that $(E^{y_j}_\la)_{j\in\Lambda}\leq (E^{x_j}_\la)_{j\in\Lambda}$ \ifff{} $E^{y_j}_\la\leq E^{x_j}_\la$ for all $j\in\Lambda$. By Lemma~\ref{spectral family of direct sum}, $(x_j)_{j\in\Lambda}\preceq (y_j)_{j\in\Lambda}$ \ifff{} $x_j\preceq y_j$ for every $j\in\Lambda$.
\end{proof}

We have pointed out in the introduction that the spectral lattice of an \AW{} \Mcal{} is a conditionally complete lattice. Let us note that suprema and infima can be described in terms of spectral families as follows. Let $M$ be a nonempty set of $\Mcal_{sa}$. If $M$ is bounded above, then its supremum is a self-adjoint element with 
\[
\left(E^{\bigvee_{x\in M} x}_\la\right)_{\la\in\rr}
=\left(\bigwedge_{x\in M}E^{x}_\la\right)_{\la\in\rr}.
\]
If $M$ is bounded below, then its infimum $\bigwedge_{x\in M}x$ is a self-adjoint element with the spectral family
\[
\left(E^{\bigwedge_{x\in M} x}_\la\right)_{\la\in\rr}
=\left(\bigwedge_{\mu>\la}\bigvee_{x\in M}E^{x}_\mu\right)_{\la\in\rr}.
\]
It was established in \cite[Proposition~3.4]{Bo21} that suprema and infima of subsets of $\Ecal(\Mcal)$ considered in the spectral lattice of $\Mcal$ are the same as those computed in spectral sublattice $(\Ecal(\Mcal),\preceq)$. In addition, it was shown that similar results hold for the sublattices $(P(\Mcal),\preceq)=(P(\Mcal),\leq)$ and $(\Mcal_+,\preceq)$ as well. 

\begin{lem}\label{infimum with central projection}
If $z$ is a central projection and $x\in\Ecal(\Mcal)$, then $zx=z\wedge x$.
\end{lem}
	\begin{proof}
Consider an abelian \AWsa{} \Ncal{} of \Mcal{} generated by $\{z,x,\unit\}$. Then $\Ncal\simeq C(X)$, where $X$ is a Stonean space. Let $f\in C(X)$ correspond to $x$. Since $x\in\Ecal(\Mcal)$, $0\leq f(t)\leq 1$ for all $t\in X$. Furthermore, there is a clopen set $O\subset X$ such that $\chi_O$ corresponds to $z$ because $z$ is a projection. The spectral order $\preceq$ coincides with $\leq$ on abelian algebras. It is easy to see that $\chi_O\wedge f=\chi_O f$. Thus $z\wedge x=zx$ in the spectral lattice of \Ncal{} because a \Sim{} is a spectral order isomorphism. According to \cite[Proposition~3.3]{Bo21}, the infimum in the spectral lattice of \Ncal{} coincides with the infimum in the spectral lattice of \Mcal{} and so $z\wedge x=zx$ in the spectral lattice of \Mcal{}.
	\end{proof}

\begin{lem}\label{supremum and multiplication}
	Let $(z_j)_{j\in\Lambda}$ be a family of mutually orthogonal central projections in an \AW{} \Mcal{}. If $x\in \Mcal_+$, then 
	$$
	\bigvee_{j\in\Lambda}z_j x=\left(\bigvee_{j\in\Lambda}z_j\right) x.
	$$
\end{lem}
\begin{proof}
	First assume that $x\in\Ecal(\Mcal)$. By Lemma~\ref{infimum with central projection},
	$$
	E^{\bigvee_{j\in\Lambda}z_j x}_\la
	=\bigwedge_{j\in\Lambda} E^{z_j x}_\la
	=\bigwedge_{j\in\Lambda} E^{z_j\wedge x}_\la
	=\bigwedge_{j\in\Lambda}\bigwedge_{\mu>\la} (E^{z_j}_\mu\vee E^{x}_\mu)
	=\bigwedge_{\mu>\la}\bigwedge_{j\in\Lambda} (E^{z_j}_\mu\vee E^{x}_\mu).
	$$	
	Since the \AWsa{} of \Mcal{} generated by 
	$$
	\set{E^{z_j}_\la}{j\in\Lambda, \la\in\rr}\cup\set{E^x_\la}{\la\in\rr}
	$$
	is abelian,
	\begin{align*}
		E^{\bigvee_{j\in\Lambda}z_j x}_\la
		&=\bigwedge_{\mu>\la} \left[\left(\bigwedge_{j\in\Lambda} E^{z_j}_\mu\right)\vee E^{x}_\mu\right]
		=\bigwedge_{\mu>\la} \left[E^{\bigvee_{j\in\Lambda} z_j}_\mu\vee E^{x}_\mu\right]\\
		&=E^{\left(\bigvee_{j\in\Lambda} z_j\right)\wedge x}_\la
		=E^{\left(\bigvee_{j\in\Lambda} z_j\right) x}_\la.
	\end{align*}
	
	Now let $x\in\Mcal_+$. Then $x=\al e$ for some $\al\in(0,\infty)$ and $e\in\Ecal(\Mcal)$. Hence
	$$
	\bigvee_{j\in\Lambda}z_j x=\al\bigvee_{j\in\Lambda}z_j e=\al \left(\bigvee_{j\in\Lambda}z_j\right) e=\left(\bigvee_{j\in\Lambda}z_j\right) x.	
	$$
\end{proof}

The goal of the following two propositions is to characterize scalar multiples of atomic projections by means of the spectral order.

\begin{pro}[{\cite[Proposition~3.7]{Bo21}}]\label{characterization of atomic projections on effects}
	Let $\Mcal$ be an \AW{} and let $x\in \Ecal(\Mcal)$ be nonzero. Then the following statements are equivalent:
	\begin{enumerate}
		\item There is $\la\in (0,1]$ and an atomic projection $e\in \Mcal$ such that $x=\la e$.
		\item If $y,z\in\Ecal(\Mcal)$ satisfy $y,z\preceq x$, then $y\preceq z$ or $z\preceq y$.
	\end{enumerate}
\end{pro}

\begin{pro}\label{characterization of atomic projections on positive operators}
	Let $\Mcal$ be an \AW{} and let $x\in \Mcal_{+}$ be nonzero. Then the following statements are equivalent:
	\begin{enumerate}
		\item There is $\al> 0$ and an atomic projection $e\in \Mcal$ such that $x=\al e$.
		\item If $y,z\in\Mcal_{+}$ satisfy $y,z\preceq x$, then $y\preceq z$ or $z\preceq y$.
	\end{enumerate}
\end{pro}
\begin{proof}
	It follows directly from the previous proposition and \cite[Lemma~3.1]{Bo21}.
\end{proof}

An element $z$ in a lattice $(P,\leq)$ is said to be {\it distributive} if 
\[
z\vee(x\wedge y)=(z\vee x)\wedge(z\vee y).
\]
The set of all distributive elements in $(P,\leq)$ is denoted by $\Dcal(P,\leq)$.
The next statement plays a fundamental role in our discussion of spectral order isomorphisms. 

\begin{pro}[{\cite[Proposition~3.8]{Bo21}}]\label{central elements}
	Let \Mcal{} be an \AW{}. Then
	\begin{enumerate}
		\item $\Zcal(\Mcal)_{sa}=\Dcal(\Mcal_{sa},\preceq)$;
		\item $\Zcal(\Mcal)_+=\Dcal(\Mcal_+,\preceq)$;
		\item $\Ecal(\Zcal(\Mcal))=\Dcal(\Ecal(\Mcal),\preceq)$.
	\end{enumerate}
\end{pro}

A bijection $\tau:P(\Mcal)\to P(\Ncal)$ between projection lattices of \AW{}s $\Mcal$ and $\Ncal$ is called a {\it projection isomorphism} if it preserves the order in both directions (i.e., for all $p,q\in P(\Mcal)$, $p\leq q$ \ifff{} $\tau(p)\leq \tau(q)$). Let $H$ and $K$ be Hilbert spaces of dimension at least 3. By the fundamental theorem of projective geometry \cite[p. 203]{Be95} and the result of Fillmore and Longstaff \cite{FL84}, the form of projective isomorphisms is well known when $\Mcal=B(H)$ and $\Ncal=B(K)$. An interesting result on projection isomorphisms was recently proved by Mori \cite{Mo20}.
He described projection isomorphisms between projection lattices of von Neumann algebras by means of ring isomorphisms between algebras of locally measurable operators.

Let $\tau:P(\Mcal)\to P(\Ncal)$ be a projection isomorphism. In the sequel, we shall denote by $\Theta_\tau$ the bijection from $\Mcal_{sa}$ onto $\Ncal_{sa}$ defined by 
$$
E^{\Theta_\tau(x)}_\la=\tau(E^x_\la).
$$
This map is indeed a spectral order isomorphism. Since  $\Theta_\tau(\Ecal(\Mcal))=\Ecal(\Ncal)$ and $\Theta_\tau(\Mcal_+)=\Ncal_+$, the corresponding restrictions of $\Theta_\tau$ are spectral order isomorphisms $\f:\Ecal(\Mcal)\to\Ecal(\Ncal)$ and $\psi:\Mcal_+\to\Ncal_+$.

An important example of a projection isomorphism is given by a restriction of \Jsi{}. By a {\it \Jsi{}} we mean a linear bijection $\psi:\Mcal\to\Ncal{}$ such that, for all $x\in\Mcal$, $\psi(x^2)=\psi(x)^2$ and $\psi(x^*)=\psi(x)^*$. If a projection isomorphism $\tau$ is a restriction of a \Jsi{} $\psi:\Mcal\to \Ncal$, then $\Theta_\tau(x)=\psi(x)$ for all $x\in\Mcal_{sa}$. As $\psi$ preserves orthogonality relation in both directions, a restriction of $\psi$ is a spectral order orthoisomorphism.

\section{Isomorphisms between lattices of effects}

In the sequel, we shall denote by $P_{at}(\Mcal)$ the set of all atomic projections in an \AW{} \Mcal{}. The proof of the following lemma is based on arguments used in \cite{MS07}.

\begin{lem}\label{preserving projections}
Let \Mcal{} and \Ncal{} be \AW{}s. If $\f:\Ecal(\Mcal)\to\Ecal(\Ncal)$ is a spectral order isomorphism, then $\f(P_{at}(\Mcal))=P_{at}(\Ncal)$.
\end{lem}
	\begin{proof}
By Proposition~\ref{characterization of atomic projections on effects}, $\f(M)=N$, where 
				\begin{align*}
					M&=\set{\la p}{p\in P_{at}(\Mcal), \la\in(0,1]},\\
					N&=\set{\la p}{p\in P_{at}(\Ncal), \la\in(0,1]}.
				\end{align*}
				The set of all maximal elements of $(M,\preceq)$ (resp. $(N,\preceq)$) is $P_{at}(\Mcal)$ (resp. $P_{at}(\Ncal)$).
				Thus $\f(P_{at}(\Mcal))=P_{at}(\Ncal)$.
	\end{proof}
	
Throughout the rest of this paper, $(\Mcal_j)_{j\in\Lambda}$ and $(\Ncal_k)_{k\in\Gamma}$ will be (nonempty) families of \AWf{}s. Let $\delta_{jk}$ be the Kronecker delta. For each $j\in\Lambda$ and $k\in\Gamma$, we set
\begin{align}
	z_j&=(\delta_{jl}\unit_{\Mcal_l})_{l\in\Lambda},\label{central elements in M}\\	w_k&=(\delta_{kl}\unit_{\Ncal_l})_{l\in\Gamma}.\label{central elements in N}
\end{align}
Note that elements $z_j$ and $w_k$ belong to the center of $\bigoplus_{j\in\Lambda}\Mcal_j$ and $\bigoplus_{k\in\Gamma}\Ncal_k$, respectively.

	\begin{theo}\label{atomic algebras and effects}
Let $\Mcal=\bigoplus_{j\in\Lambda}\Mcal_j$ and $\Ncal=\bigoplus_{k\in\Gamma}\Ncal_k$, where $\Mcal_j$ and $\Ncal_k$ are \AWf{}s. If $\Phi:\Ecal(\Mcal)\to\Ecal(\Ncal)$ is a spectral order isomorphism, then there are a bijection $\pi:\Gamma\to\Lambda$ and a family $(\f_j)_{j\in\Lambda}$ of spectral order isomorphisms $\f_j:\Ecal(\Mcal_j)\to\Ecal(\Ncal_{\pi^{-1}(j)})$ such that 
$$
\Phi((x_j)_{j\in\Lambda})=(\f_{\pi(k)}(x_{\pi(k)}))_{k\in\Gamma}.
$$
	\end{theo}
		\begin{proof}
Using Proposition~\ref{central elements}, we see that $\Phi(\Ecal(\Zcal(\Mcal)))=\Ecal(\Zcal(\Ncal))$. Since $\Zcal(\Mcal)=\bigoplus_{j\in\Lambda}\Zcal(\Mcal_j)$ and each $\Zcal(\Mcal_j)$ is \Sic{} to \cc{}, $\Zcal(\Mcal)$ is an atomic \AW{}. Similarly, $\Zcal(\Ncal)$ is an atomic \AW{}. Furthermore, Lemma~\ref{preserving projections} establishes that \[\Phi(P_{at}(\Zcal(\Mcal)))=P_{at}(\Zcal(\Ncal)).\]

Let $z_j$ and $w_k$ be elements defined in (\ref{central elements in M}) and (\ref{central elements in N}), respectively. Clearly,
\[
P_{at}(\Zcal(\Mcal))=\set{z_j}{j\in\Lambda}\quad \mbox{and}\quad
P_{at}(\Zcal(\Ncal))=\set{w_k}{k\in\Gamma}.
\]
It follows from $\Phi(P_{at}(\Zcal(\Mcal)))=P_{at}(\Zcal(\Ncal))$ that there is a bijection $\pi:\Gamma\to\Lambda$ such that $\Phi(z_j)=w_{\pi^{-1}(j)}$ for all $j\in\Lambda$. If $x\in z_j\Ecal(\Mcal)$, then
\[
\Phi(x)=\Phi(z_j x) =\Phi(z_j\wedge x)=w_{\pi^{-1}(j)}\wedge \Phi(x)=w_{\pi^{-1}(j)}\Phi(x)\in w_{\pi^{-1}(j)}\Ecal(\Ncal)
\]
by Lemma~\ref{infimum with central projection}. On the other hand, if $y\in w_{\pi^{-1}(j)}\Ecal(\Ncal)$, then
\[
\Phi(z_j\Phi^{-1}(y))=\Phi(z_j\wedge\Phi^{-1}(y))=w_{\pi^{-1}(j)}\wedge y=y.
\]
This shows that $\Phi(z_j\Ecal(\Mcal))=w_{\pi^{-1}(j)}\Ecal(\Ncal)$ for all $j\in\Lambda$. In other words, there is a family $(\f_j)_{j\in\Lambda}$ of spectral order isomorphisms $\f_j:\Ecal(\Mcal_j)\to\Ecal(\Ncal_{\pi^{-1}(j)})$ such that 
\[
\Phi(z_l(x_k)_{k\in\Lambda})=w_{\pi^{-1}(l)}(\f_{\pi(k)}(x_{\pi(k)}))_{k\in\Gamma}
\]
for all $l\in\Lambda$. According to Lemma~\ref{supremum and multiplication},
\begin{align*}
	\Phi((x_j)_{j\in\Lambda})
	&=\Phi(\bigvee_{l\in\Lambda}z_l (x_j)_{j\in\Lambda})
	=\bigvee_{l\in\Lambda} \Phi(z_l(x_j)_{j\in\Lambda})
	=\bigvee_{l\in\Lambda} w_{\pi^{-1}(l)}(\f_{\pi(k)}(x_{\pi(k)}))_{k\in\Gamma}\\
	&=\left(\bigvee_{l\in\Lambda} w_{\pi^{-1}(l)}\right)(\f_{\pi(k)}(x_{\pi(k)}))_{k\in\Gamma}
	=(\f_{\pi(k)}(x_{\pi(k)}))_{k\in\Gamma}.
\end{align*}
		\end{proof}

It was shown in \cite[Corollary~4.2]{Bo21} (see also \cite{MN16, MS07} for the special case of automorphisms) that if $\f:\Ecal(\Mcal)\to\Ecal(\Ncal)$ is a spectral order isomorphism, where \Mcal{} and \Ncal{} are \AWf{}s of Type~I, then there are a projection isomorphism $\tau:P(\Mcal)\to P(\Ncal)$ and a bijection $f:[0,1]\to[0,1]$ such that $\f(x)=\Theta_\tau(f(x))$ for all $x\in\Ecal(\Mcal)$. With this fact in mind, Theorem~\ref{atomic algebras and effects} leads immediately to a complete description of spectral order isomorphisms between spectral sublattices of all effects of direct sums of Type~I factors. We formulate an explicit statement in the case of direct sums of full matrix algebras. Of course, the special case of general matrix algebras is covered by this result.

\begin{cor}\label{spectral order isomorphisms on effects in finite dimension}
Let $\Mcal=\bigoplus_{j\in \Lambda} B(\cc^{m_j})$ and $\Ncal=\bigoplus_{k\in \Gamma} B(\cc^{n_k})$, where $m_j$ and $n_k$ are natural numbers for each $j\in \Lambda$ and each $k\in \Gamma$. If $\Phi:\Ecal(\Mcal)\to\Ecal(\Ncal)$ is a spectral order isomorphism, then there are a bijection $\pi:\Gamma\to \Lambda$ with $m_{j}=n_{\pi^{-1}(j)}$ for all $j\in \Lambda$, a family $(f_j)_{j\in \Lambda}$ of strictly increasing bijections $f_j:[0,1]\to[0,1]$, and a family $(\tau_j)_{j\in \Lambda}$ of projection automorphisms $\tau_j:P(B(\cc^{m_j}))\to P(B(\cc^{m_j}))$ such that 
\[
\Phi((x_j)_{j\in \Lambda})=(\Theta_{\tau_{\pi(k)}}(f_{\pi(k)}(x_{\pi(k)})))_{k\in \Gamma}.
\]
\end{cor}
	\begin{proof}		
Theorem~\ref{atomic algebras and effects} together with \cite[Corollary~4.2]{Bo21} ensure the existence of a bijection $\pi:\Gamma\to \Lambda$, a family $(f_j)_{j\in \Lambda}$ of strictly increasing bijections $f_j:[0,1]\to[0,1]$, and a family $(\tau_j)_{j\in \Lambda}$ of projection isomorphisms $\tau_j:\Ecal(B(\cc^{m_j}))\to\Ecal(B(\cc^{n_{\pi^{-1}(j)}}))$ such that 
\[
\Phi((x_j)_{r\in \Lambda})=(\Theta_{\tau_{\pi(k)}}(f_{\pi(k)}(x_{\pi(k)})))_{k\in \Gamma}.
\]
As there is a projection isomorphism $\tau_r:\Ecal(B(\cc^{m_r}))\to\Ecal(B(\cc^{n_{\pi^{-1}(r)}}))$, we have $m_r=n_{\pi^{-1}(r)}$.
	\end{proof}

The next simple consequence of Theorem~\ref{atomic algebras and effects} is the following description of spectral order orthoisomorphisms. 

\begin{cor}
Suppose that $\Mcal=\bigoplus_{j\in\Lambda}\Mcal_j$ and $\Ncal=\bigoplus_{k\in\Gamma}\Ncal_k$, where \AWf{}s $\Mcal_j$ and $\Ncal_k$ are not of  Type~I$_2$. If $\Phi:\Ecal(\Mcal)\to\Ecal(\Ncal)$ is a spectral order orthoisomorphism, then there are a bijection $\pi:\Gamma\to\Lambda$, a family $(\psi_j)_{j\in\Lambda}$ of \Jsi{}s $\psi_j:\Mcal_j\to\Ncal_{\pi^{-1}(j)}$, and a family $(f_j)_{j\in\Lambda}$ of strictly increasing bijections $f_j:[0,1]\to[0,1]$ such that 
	\[
	\Phi((x_j)_{j\in\Lambda})=(\psi_{\pi(k)}(f_{\pi(k)}(x_{\pi(k)})))_{k\in\Gamma}.
	\]	
\end{cor}
\begin{proof}
	The proof follows immediately from \cite[Corollary~5.2]{Bo21} and Theorem~\ref{atomic algebras and effects}.
\end{proof}
\section{Isomorphisms between lattices of positive elements}

\begin{theo}\label{atomic algebras and positive operators}
Let $\Mcal=\bigoplus_{j\in\Lambda}\Mcal_j$ and $\Ncal=\bigoplus_{k\in\Gamma}\Ncal_k$, where $\Mcal_j$ and $\Ncal_k$ are \AWf{}s. If $\Phi:\Mcal_+\to\Ncal_+$ is a spectral order isomorphism, then there are a bijection $\pi:\Gamma\to\Lambda$ and a family $(\f_j)_{j\in\Lambda}$ of spectral order isomorphisms $\f_j$ from the positive part of $\Mcal_j$ onto positive part of $\Ncal_{\pi^{-1}(j)}$ such that 
\[
\Phi((x_j)_{j\in\Lambda})=(\f_{\pi(k)}(x_{\pi(k)}))_{k\in\Gamma}.
\]	
\end{theo}
	\begin{proof}
It follows from Proposition~\ref{central elements} that $\Phi(\Zcal(\Mcal)_+)=\Zcal(\Ncal)_+$. Thus a restriction of $\Phi$ is a spectral order isomorphism from $\Zcal(\Mcal)_+$ onto $\Zcal(\Ncal)_+$. Since $\Mcal$ and $\Ncal$ are direct sums of factors, $\Zcal(\Mcal)$ and $\Zcal(\Ncal)$ are atomic \AW{}s. Let $z_j$ and $w_k$ be elements defined in (\ref{central elements in M}) and (\ref{central elements in N}), respectively. Suppose that $j\in\Lambda$. By Proposition~\ref{characterization of atomic projections on positive operators}, for each $\la\in(0,\infty)$, there are $k\in\Gamma$ and a positive number $f_j(\la)$ such that $\Phi(\la z_j)=f_j(\la)w_k$. We show that $k$ does not depend on \la{}. For this, choose $\la,\mu\in(0,\infty)$. Then 
\[
f_j(\max\{\la,\mu\})w_m=\Phi(\max\{\la,\mu\}z_j)=\Phi(\la z_j\vee \mu z_j)=f_j(\la)w_k\vee f_j(\mu)w_l
\]
for some $k,l,m\in\Gamma$. If $k\neq l$, then we see from \cite[Lemma~3.5]{Bo21} that $f_j(\la)=0$ or $f_j(\mu)=0$ which is a contradiction. Consequently,
there are a permutation $\pi:\Gamma\to\Lambda$ and a family $(f_j)_{j\in\Lambda}$ of strictly increasing bijections $f_j:[0,\infty)\to[0,\infty)$ such that
\[
\Phi(\la z_j)=f_j(\la)w_{\pi^{-1}(j)}
\]
for all $\la\in[0,\infty)$ and all $j\in\Lambda$.

Let $(x_j)_{j\in\Lambda}$ belong to $\Mcal_+$. Then there is $c\in(0,\infty)$ such that $x_j\preceq c\unit_{\Mcal_j}$ for all $j\in\Lambda$. We observe that, for each $j\in\Lambda$, 
\[
f^{-1}_j(c)z_j=\Phi^{-1}(cw_{\pi^{-1}(j)})\preceq\Phi^{-1}(c\unit_{\Ncal})
\] 
and so $0\leq f^{-1}_j(c)z_j\leq\Phi^{-1}(c\unit_{\Ncal})$. Hence 

\[
\sup_{j\in\Lambda}f^{-1}_j(c)=\sup_{j\in\Lambda} \norm{f^{-1}_j(c)z_j}\leq \norm{\Phi^{-1}(c\unit_{\Ncal})}<\infty.
\]
As $f^{-1}_j$ defines a spectral order automorphism of the positive part of $\Mcal_j$, 
\[f^{-1}_j(x_j)\preceq f^{-1}_j(c)\unit_{\Mcal_j}\preceq \norm{\Phi^{-1}(c\unit_{\Ncal})}\unit_{\Mcal_j}
\]
which implies that 
$$
\sup_{j\in\Lambda}\norm{f^{-1}_j(x_j)}\leq \norm{\Phi^{-1}(c\unit_{\Ncal})}<\infty.
$$ 
This allows us to conclude that $(f^{-1}_j(x_j))_{j\in\Lambda}\in\Mcal_+$. Therefore, a map $\Psi:\Mcal_+\to\Ncal_+$ given by 
$$
\Psi((x_j)_{j\in\Lambda})=\Phi((f^{-1}_j(x_j))_{j\in\Lambda})
$$ 
for all $(x_j)_{j\in\Lambda}\in\Mcal_+$ is well defined. It is easy to see that $\Psi$ is a spectral order isomorphism. Moreover, $\Psi(\la z_j)=\la w_{\pi^{-1}(j)}$ for all $j\in \Lambda$ and all $\la\in[0,\infty)$.

Suppose that $x\in z_j\Mcal_+$. Then there are $\al\in(0,\infty)$ and $e\in z_j\Ecal(\Mcal)$ such that $x=\al e$. Using Lemma~\ref{infimum with central projection} and the fact that $\Psi$ is a spectral order isomorphism,
$$
\Psi(x)=\Psi(\al (z_j \wedge e))=\Psi((\al z_j) \wedge(\al e))=(\al w_{\pi^{-1}(j)})\wedge \Psi(x)=\al\left[w_{\pi^{-1}(j)}\wedge\frac{1}{\al}\Psi(x)\right].
$$
As 
$$
\Psi(\al\unit_\Mcal)=\Psi(\bigvee_{j\in\Lambda}\al z_j)=\bigvee_{j\in\Lambda}\Psi(\al z_j)=\bigvee_{j\in\Lambda}\al w_{\pi^{-1}(j)}=\al\unit_\Ncal,
$$
we have $\frac{1}{\al}\Psi(x)\in\Ecal(\Ncal)$. It follows from Lemma~\ref{infimum with central projection} that $\Psi(x)=w_{\pi^{-1}(j)}\Psi(x)$. Similarly, we show that  $\Psi^{-1}(y)\in z_j\Mcal_+$ whenever $y\in w_{\pi^{-1}(j)}\Ncal_+$. This proves that $\Psi(z_j\Mcal_+)=w_{\pi^{-1}(j)}\Ncal_+$ for all $j\in\Lambda$ and so 
\[
\Psi(z_j(x_k)_{k\in\Lambda})=w_{\pi^{-1}(j)}(\psi_{\pi(k)}(x_{\pi(k)}))_{k\in\Gamma}
\]
for some family $(\psi_j)_{j\in\Lambda}$ of spectral order isomorphisms $\psi_j$ from the positive part of $\Mcal_{j}$ onto the positive part of $\Ncal_{\pi^{-1}(j)}$. By Lemma~\ref{supremum and multiplication},
\begin{align*}
	\Psi((x_k)_{k\in\Lambda})
	&=\Psi(\bigvee_{j\in\Lambda}z_j (x_k)_{k\in\Lambda})
	=\bigvee_{j\in\Lambda} \Psi(z_j(x_k)_{k\in\Lambda})
	=\bigvee_{j\in\Lambda} w_{\pi^{-1}(j)}(\psi_{\pi(k)}(x_{\pi(k)}))_{k\in\Gamma}\\
	&=\left(\bigvee_{j\in\Lambda} w_{\pi^{-1}(j)}\right)(\psi_{\pi(k)}(x_{\pi(k)}))_{k\in\Gamma}
	=(\psi_{\pi(k)}(x_{\pi(k)}))_{k\in\Gamma}.
\end{align*}
	\end{proof}

As in the case of effects, one can formulate the following corollaries.

\begin{cor}
Let $\Mcal=\bigoplus_{j\in \Lambda} B(\cc^{m_j})$ and $\Ncal=\bigoplus_{k\in \Gamma} B(\cc^{n_k})$, where $m_j$ and $n_k$ are natural numbers for each $j\in \Lambda$ and each $k\in \Gamma$. If $\Phi:\Mcal_+\to\Ncal_+$ is a spectral order isomorphism, then there are a bijection $\pi:\Gamma\to \Lambda$ with $m_{j}=n_{\pi^{-1}(j)}$ for all $j\in \Lambda$, a family $(f_j)_{j\in \Lambda}$ of strictly increasing bijections $f_j:[0,\infty)\to[0,\infty)$, and a family $(\tau_j)_{j\in \Lambda}$ of projection automorphisms $\tau_j:P(B(\cc^{m_j}))\to P(B(\cc^{m_j}))$ such that 
\[
\Phi((x_j)_{j\in \Lambda})=(\Theta_{\tau_{\pi(k)}}(f_{\pi(k)}(x_{\pi(k)})))_{k\in \Gamma}.
\]
\end{cor}
\begin{proof}
	The corollary is a direct consequence of \cite[Theorem~4.5]{Bo21} and Theorem~\ref{atomic algebras and positive operators}.
\end{proof}

\begin{cor}
	Suppose that $\Mcal=\bigoplus_{j\in\Lambda}\Mcal_j$ and $\Ncal=\bigoplus_{k\in\Gamma}\Ncal_k$, where \AWf{}s $\Mcal_j$ and $\Ncal_k$ are not of  Type~I$_2$. If $\Phi:\Mcal_+\to\Ncal_+$ is a spectral order orthoisomorphism, then there are a bijection $\pi:\Gamma\to\Lambda$, a family $(\psi_j)_{j\in\Lambda}$ of \Jsi{}s $\psi_j:\Mcal_j\to\Ncal_j$, and a family $(f_j)_{j\in\Lambda}$ of strictly increasing bijections $f_j:[0,\infty)\to [0,\infty)$ such that 
	\[
	\Phi((x_j)_{j\in\Lambda})=(\psi_{\pi(k)}(f_{\pi(k)}(x_{\pi(k)})))_{k\in\Gamma}.
	\]	
\end{cor}
\begin{proof}
	The proof follows immediately from \cite[Corollary~5.4]{Bo21} and Theorem~\ref{atomic algebras and positive operators}.
\end{proof}

\section{Isomorphisms between spectral lattices}

Let us fix a notation. By $x^+$ and $x^-$ we denote, respectively, the positive part and the negative part of a self-adjoint element $x$ in an \AW{}. First, we recall a useful lemma proved in \cite{Bo21}.

\begin{lem}[{\cite[Lemma~5.5]{Bo21}}]\label{positive and negative parts}
	Let $\f:\Mcal\to\Ncal$ be a spectral order isomorphism between \AW{}s \Mcal{} and \Ncal{} with $\f(0)=0$. If $x\in\Mcal_{sa}$, then $\f(x)^+=\f(x^+)$ and $\f(x)^-=-\f(-(x^-))$.
\end{lem}

We can now state the main result of this paper. 

\begin{theo}\label{atomic algebras and self-adjoint operators}
	Let $\Mcal=\bigoplus_{j\in\Lambda}\Mcal_j$ and $\Ncal=\bigoplus_{k\in\Gamma}\Ncal_k$, where $\Mcal_j$ and $\Ncal_k$ are \AWf{}s. If $\Phi:\Mcal_{sa}\to\Ncal_{sa}$ is a spectral order isomorphism, then there are a bijection $\pi:\Gamma\to\Lambda$ and a family $(\f_j)_{j\in\Lambda}$ of spectral order isomorphisms $\f_j$ from the self-adjoint part of $\Mcal_j$ onto self-adjoint part of $\Ncal_{\pi^{-1}(j)}$ such that 
	\[
	\Phi((x_j)_{j\in\Lambda})=(\f_{\pi(k)}(x_{\pi(k)}))_{k\in\Gamma}.
	\]	
\end{theo}
\begin{proof} 
By Proposition~\ref{central elements}, $\Phi$ restricts to a spectral order isomorphism from the atomic \AW{} $\Zcal(\Mcal)_{sa}$ onto the atomic \AW{} $\Zcal(\Ncal)_{sa}$. Combining Lemma~\ref{spectral order on direct sum} with \cite[Lemma~3.1]{Bo21}, we see that the map 
\[x\mapsto \Phi(x)-\Phi(0)\]
is a spectral order isomorphism from $\Mcal_{sa}$ onto $\Ncal_{sa}$ because $\Phi(0)$ is a central element and so $\Phi(0)=(\al_j\unit_{\Ncal_j})$ for some family $(\al_j)_{j\in\Lambda}$ of real numbers. Therefore, we can assume without loss of generality that $\Phi(0)=0$. Using this assumption, we have $\Phi(\Mcal_+)=\Ncal_+$ and $\Phi(\Mcal_-)=\Ncal_-$, where $\Mcal_-=-\Mcal_+$ and $\Ncal_-=-\Ncal_+$. By Theorem~\ref{atomic algebras and positive operators} and \cite[Theorem~4.5]{Bo21}, there are a family $(g_j)_{j\in\Lambda}$ of strictly increasing bijections $g_j:[0,\infty)\to[0,\infty)$ and a permutation $\pi:\Gamma\to\Lambda$ such that 
\[
\Phi(\la z_j)=g_j(\la)w_{\pi^{-1}(j)}
\]
for all $\la\in[0,\infty)$ and all $j\in\Lambda$, where $z_j$ and $w_k$ are elements defined in (\ref{central elements in M}) and (\ref{central elements in N}), respectively.

Set $\Psi(x)=-\Phi(-x)$, $x\in\Mcal_{sa}$. Then $\Psi:\Mcal_{sa}\to\Ncal_{sa}$ is a spectral order isomorphism from $\Mcal_{sa}$ onto $\Ncal_{sa}$ because the multiplication by $-1$ is order-reversing. Furthermore, $\Psi(0)=0$, $\Psi(\Mcal_+)=\Ncal_+$, and $\Psi(\Mcal_-)=\Ncal_-$. If $\la\leq 0$, then we conclude from the above discussion that $\Psi(\la z_j)=-g_j(-\la)w_{\pi^{-1}(j)}$ for each $j\in\Lambda$. If $\la\geq 0$, then we obtain from Theorem~\ref{atomic algebras and positive operators}, \cite[Theorem~4.5]{Bo21}, and $\Psi(\Zcal(\Mcal)_{sa})=\Zcal(\Ncal)_{sa}$
that there are a family $(h_j)_{j\in\Lambda}$ of strictly increasing bijections and a permutation $\sigma:\Gamma\to\Lambda$ such that $\Psi(\la z_j)=h_j(\la)w_{\sigma^{-1}(j)}$. We are going to show that $\pi=\sigma$. To prove this we suppose that $\pi\neq \sigma$. Then there are two different indices $k,l\in\Lambda$  such that $\pi^{-1}(l)=\sigma^{-1}(k)$. Applying Lemma~\ref{positive and negative parts},
	$$
	\Psi(z_k-z_l)=\Psi(z_k)+\Psi(-z_l)=h_k(1)w_{\sigma^{-1}(k)}-g_l(1)w_{\pi^{-1}(l)}=\left(h_k(1)-g_l(1)\right)w_{\sigma^{-1}(k)}.
	$$
	However, $\left(h_k(1)-g_l(1)\right)w_{\sigma^{-1}(k)}$ is in $\Ncal_+$ or $\Ncal_-$ which is a contradiction because $z_k-z_l$ does not belong to $\Mcal_+$ or $\Mcal_-$. Therefore, $\pi=\sigma$. As a consequence, $\Psi(\la z_j)=f_j(\la)w_{\pi^{-1}(j)}$, where 
	\[
	f_j(\la)=\begin{cases}
	&h_j(\la), \quad \mbox{if } \la\geq 0;\\
	&-g_j(-\la)\quad \mbox{if } \la< 0.
	\end{cases}
	\]
	Note that $f_j:\rr\to\rr$ is a strictly increasing bijection.
	
	If $(x_j)_{j\in\Lambda}\in\Mcal_{sa}$, then there is $c\in(0,\infty)$ such that $-c\unit_{\Mcal_j}\preceq x_j\preceq c\unit_{\Mcal_j}$ for each $j\in\Lambda$. Arguments used in the proof of Theorem~\ref{atomic algebras and positive operators} establish that 
	$\sup_{j\in\Lambda}f_j^{-1}(c)\leq \norm{\Psi^{-1}(c\unit_\Ncal)}$.	We can prove similarly that $\sup_{j\in\Lambda} -f_j^{-1}(-c)\leq \norm{\Psi^{-1}(-c\unit_\Ncal)}$. Put 
	\[
	\al=\max\left\{\norm{\Psi^{-1}(c\unit_\Ncal)},\norm{\Psi^{-1}(-c\unit_\Ncal)}\right\}.
	\]
	We deduce from $-c\unit_{\Mcal_j}\preceq x_j\preceq c\unit_{\Mcal_j}$ that 
	$$
	-\al\unit_\Mcal \preceq f_j^{-1}(-c)\unit_{\Mcal_j}\preceq f_j^{-1}(x_j)\preceq f_j^{-1}(c)\unit_{\Mcal_j}\preceq \al\unit_\Mcal
	$$
	for all $j\in\Lambda$. Accordingly, 
	$$
	\sup_{j\in\Lambda}\norm{f^{-1}_j(x_j)}\leq \norm{\al\unit_\Mcal}<\infty
	$$ 
	Thus $(x_j)_{j\in\Lambda}\mapsto \Psi((f^{-1}_j(x_j))_{j\in\Lambda})$ is a well defined spectral order isomorphism from $\Mcal_{sa}$ onto $\Ncal_{sa}$ and so we can assume without loss of generality that $f_j$ is the identity function for each $j\in\Lambda$. In other words, we shall suppose that 
	$$
	\Psi(\la z_{j})=\la w_{\pi^{-1}(j)}
	$$
	for all $\la\in\rr$ and all $j\in\Lambda$.
	
	To complete the proof of our assertion it is sufficient to show that there is a family $(\psi_j)_{j\in\Lambda}$ of spectral order isomorphisms $\psi_j$ from the self-adjoint part of $\Mcal_j$ onto self-adjoint part of $\Ncal_{\pi^{-1}(j)}$ such that 
	$$
	\Psi((x_j)_{j\in\Lambda})=(\psi_{\pi(k)}(x_{\pi(k)}))_{k\in\Gamma}.
	$$	
	for every $(x_j)_{j\in\Lambda}\in\Mcal_{sa}$. If $x\in z_j\Mcal_{sa}$, then $x^+$ and $x^-$ belong to $z_j\Mcal_{sa}$. We can write $x^+$ and $x^-$ in the form $x^+=\al u$ and $x^-=\beta v$ for some $\al,\beta\in(0,\infty)$ and $u,v\in z_j\Ecal(\Mcal)$. 
	It was pointed out in the proof of Theorem~\ref{atomic algebras and positive operators} that $\frac{1}{\al}\Psi(x^+)\in \Ecal(\Ncal)$. We can apply a similar reasoning to prove that $-\frac{1}{\beta}\Psi(-x^-)\in\Ecal(\Ncal)$. Therefore, we obtain from Lemma~\ref{infimum with central projection} and Lemma~\ref{positive and negative parts} that 
	\begin{align*}
		\Psi(x)&=\Psi(x)^+ - \Psi(x)^-=\Psi(x^+) + \Psi(-x^-)=\Psi(\al(z_j\wedge u))+\Psi(-\beta(z_j\wedge v))\\
		&=\Psi(\al z_j)\wedge \Psi(x^+)+\Psi(-\beta z_j)\vee\Psi(-x^-)\\
		&=\al\left[w_{\pi^{-1}(j)}\wedge \frac{1}{\al}\Psi(x^+)\right]-\beta\left[w_{\pi^{-1}(j)}\wedge\left(-\frac{1}{\beta}\Psi(-x^-)\right)\right]\\
		&= w_{\pi^{-1}(j)}\left[\Psi(x^+)+\Psi(-x^-)\right]=w_{\pi^{-1}(j)}\Psi(x).
	\end{align*}
	We also observe from analogous arguments that $\Psi^{-1}(y)\in z_j\Mcal_{sa}$ whenever $y\in w_{\pi^{-1}(j)}\Ncal_{sa}$. This means that $\Psi(z_j\Mcal_{sa})=w_{\pi^{-1}(j)}\Ncal_{sa}$ for all $j\in\Lambda$. Thus there exists a family $(\psi_j)_{j\in\Lambda}$ of spectral order isomorphisms from the self-adjoint part of $\Mcal_j$ onto the self-adjoint part of $\Ncal_{\pi^{-1}(j)}$ such that 
	\[
	\Psi(z_j(x_k)_{k\in\Lambda})=w_{\pi^{-1}(j)}(\psi_{\pi(k)}(x_{\pi(k)}))_{k\in\Gamma}
	\]
	whenever $(x_k)_{k\in\Lambda}\in\Mcal_{sa}$.  Taking into account Lemma~\ref{supremum and multiplication} together with Lemma~\ref{positive and negative parts}, we get
	\begin{align*}
		\Psi((x_k)_{k\in\Lambda})&= \Psi((x_k^+)_{k\in\Lambda})+\Psi(-(x_k^-)_{k\in\Lambda}) = \Psi(\bigvee_{j\in\Lambda} z_j (x_k^+)_{k\in\Lambda})+\Psi(-\bigvee_{j\in\Lambda} z_j (x_k^-)_{k\in\Lambda})\\
		&=\bigvee_{j\in\Lambda} \Psi( z_j (x_k^+)_{k\in\Lambda})+\bigwedge_{j\in\Lambda} \Psi(- z_j (x_k^-)_{k\in\Lambda})\\
		&=\bigvee_{j\in\Lambda} w_{\pi^{-1}(j)}(\psi_{\pi(k)}(x_{\pi(k)}^+))_{k\in\Gamma}-\bigvee_{j\in\Lambda} w_{\pi^{-1}(j)}(-\psi_{\pi(k)}(-x_{\pi(k)}^-))_{k\in\Gamma}\\
		&= (\psi_{\pi(k)}(x_{\pi(k)})^{+}-\psi_{\pi(k)}(x_{\pi(k)})^-)_{k\in\Gamma}= (\psi_{\pi(k)}(x_{\pi(k)}))_{k\in\Gamma}.
	\end{align*}
	for all $(x_k)_{k\in\Lambda}\in\Mcal_{sa}$.
\end{proof}

For the sake of completeness, we state some consequences of the previous theorem. The first result is concerned with spectral order isomorphisms between spectral lattices of direct sum of full matrix algebras. The second assertion describes spectral order orthoisomorphisms between spectral lattices of direct sum of \AWf{}.

\begin{cor}
	Let $\Mcal=\bigoplus_{j\in \Lambda} B(\cc^{m_j})$ and $\Ncal=\bigoplus_{k\in \Gamma} B(\cc^{n_k})$, where $m_j$ and $n_k$ are natural numbers for each $j\in \Lambda$ and each $k\in \Gamma$. If $\Phi:\Mcal_{sa}\to\Ncal_{sa}$ is a spectral order isomorphism, then there are a bijection $\pi:\Gamma\to \Lambda$ with $m_{j}=n_{\pi^{-1}(j)}$ for all $j\in \Lambda$, a family $(f_j)_{j\in \Lambda}$ of strictly increasing bijections $f_j:\rr\to\rr$, and a family $(\tau_j)_{j\in \Lambda}$ of projection automorphisms $\tau_j:P(B(\cc^{m_j}))\to P(B(\cc^{m_j}))$ such that 
	\[
	\Phi((x_j)_{j\in \Lambda})=(\Theta_{\tau_{\pi(k)}}(f_{\pi(k)}(x_{\pi(k)})))_{k\in \Gamma}.
	\]
\end{cor}
\begin{proof}
	It follows from \cite[Theorem~4.7]{Bo21} and Theorem~\ref{atomic algebras and self-adjoint operators}.
\end{proof}

\begin{cor}
	Suppose that $\Mcal=\bigoplus_{j\in\Lambda}\Mcal_j$ and $\Ncal=\bigoplus_{k\in\Gamma}\Ncal_k$, where \AWf{}s $\Mcal_j$ and $\Ncal_k$ are not of  Type~I$_2$. If $\Phi:\Mcal_{sa}\to\Ncal_{sa}$ is a spectral order orthoisomorphism, then there are a bijection $\pi:\Gamma\to\Lambda$, a family $(\psi_j)_{j\in\Lambda}$ of \Jsi{}s $\psi_j:\Mcal_j\to\Ncal_j$, and a family $(f_j)_{j\in\Lambda}$ of strictly increasing bijections $f_j:\rr\to\rr$ such that 
	\[
	\Phi((x_j)_{j\in\Lambda})=(\psi_{\pi(k)}(f_{\pi(k)}(x_{\pi(k)})))_{k\in\Gamma}.
	\]	
\end{cor}
\begin{proof}
	This follows directly from \cite[Corollary~5.7]{Bo21} and Theorem~\ref{atomic algebras and self-adjoint operators}.
\end{proof}
\section*{Acknowledgement}
This work was supported by the project OPVVV Center for Advanced Applied Science CZ.02.1.01/0.0/0.0/16\_019/0000778. 

\end{document}